\documentclass[a4paper,twoside,10pt]{article}
\usepackage{pstricks,pst-node,pst-text,pst-3d}
\usepackage{epsfig,amsmath,amsfonts}
\usepackage{hyperref}
\usepackage{cleveref}
\usepackage{pstricks}
\usepackage{float}
\usepackage{pdfsync}
\usepackage{mathtools}
\usepackage[scale=0.75]{geometry}
\usepackage{amsmath,amssymb}
\newcommand{\lrhup}[2]{%
  \ooalign{$\m@th#1\leftharpoonup$\cr$\m@th#1\rightharpoonup$\cr}%
}

\makeatother
\newtheorem{definition}{Definition}[section]
\newtheorem{lemma}[definition]{Lemma}

\newtheorem{remark}[definition]{Remark}

 \usepackage{hyperref} 
\usepackage{graphicx}
\usepackage{color}
  \usepackage{epsfig}
  \usepackage{pdfsync} % shift+z+mela +click (solo con Pdftex)!!
 
\usepackage{amssymb}
\usepackage{epsfig}
\usepackage{latexsym}
\usepackage{amssymb}
\usepackage{amsmath}
\usepackage{amsfonts}
\usepackage{srcltx}
\usepackage {amsbsy}
\def\eproof{\hfill $\Box$}

\def\R{{{\rm I}\!{\rm R}}}

\def\m{{\bf m}}

\def\u{{\bf u}}
\def\x{{\bf x}}

\def\C0{\mathcal{C}_0}
\def\Ca0{\mathcal{C}_{0,a}}

\usepackage{caption}
\usepackage[us,12hr]{datetime}                    %
\usepackage{fancyhdr}                                  %  EXTRA LINES TO OBTAIN
\begin{document}
\title{\bf Magneto-elasticity on the disk
%\\
%\hfill \\
%\red\boxed{\huge\bf ~ready ~~for~~ submission~}
}
\author{}

%\date{%May 29, 2019}%{\today}
 \maketitle

{\small
\begin{center}
 {\sc Sandra Carillo} \\
Dipartimento di Scienze di Base e Applicate \\
per l'Ingegneria, 
 Universit\`{a}   di Roma  {\textsc{La Sapienza}} ,\\
 Via Antonio Scarpa 16,  00161 Rome,  
Italy\\
\& \\
I.N.F.N. - Sezione Roma1,
Gr. IV - M.M.N.L.P.,  Rome, Italy\\
[10pt]
 {\sc Michel Chipot} \\
IMath,  University of Z\"urich\\
Winterthurerstrasse 190, 8057, Z\"urich,  Switzerland \\
[10pt]
 {\sc Vanda Valente} \\
Istituto per le Applicazioni del Calcolo {\it
M. Picone} \\

  Via dei Taurini 19, 00185 Rome, Italy  \\[10pt]
 {\sc Giorgio Vergara Caffarelli} \\
Dipartimento di Scienze di Base e Applicate \\
per l'Ingegneria, 
 Universit\`{a}   di Roma  {\textsc{La Sapienza}} ,\\
 Via Antonio Scarpa 16,  00161 Rome,  
\end{center}
}

\numberwithin{equation}{section}
\allowdisplaybreaks
 \noindent{\bf keywords}:{ magnetic field, elastic body, magneto-elastic energy, functional minimisation}

\begin{abstract}
A model problem of magneto-elastic body is considered. Specifically, the case of a two dimensional circular disk is studied. The functional which represents the magneto-elastic
energy is introduced. Then, the minimisation problem, referring to the simplified two-dimensional model under investigation, is analysed. The existence of a minimiser is 
proved and its dependence on the eigenvalues of the problem is investigated.
 A bifurcation result  is obtained corresponding to special values of the parameters.
\end{abstract}
\vfill\eject
\section{Introduction}
The interest in magneto-elastic materials finds its motivation in the growing 
variety of new materials among which magneto-rheological elastomers or   
magneto-sensitive polymeric composites \cite{Hossain-et-al-2015a, Hossain-et-al-2015b} 
may be mentioned.  A whole Special Issue devoted to {\it Magnetoelastic Materials}
is going to be published soon \cite{SI2020} in the Journal {\it Materials}. 
Many applications of magneto-elastic materials, covering a wide area of interest  from 
technological to biomedical devices, see e.g. \cite{Ren},
can be listed. In particular, also 
two dimensional problems are subject of applicative investigations \cite{Hadda}.
The model we consider is a two dimensional simplified one, however, we believe that, it might 
open the way to further applications, possibly, via perturbative methods \cite{Bernard}.

We study the functional energy of a magneto-elastic material, that is
a material which is capable of deformation and magnetisation. The
magnetisation is a phenomenon that does not appear at a macroscopic
level, it is characterised by the {\it magnetisation vector} whose magnitude 
is independent of the position while its direction which can vary from one point to another.\\
In this context,  the magnetisation vector $\m$ is a map from $\Omega$
(a bounded open set of $\mathbb{R}^2$) to $S^2$ (the unit sphere of
$\mathbb{R}^3$). In particular, here we assume $\Omega$ is the unit
disk of $\mathbb{R}^2$. The magnetisation distribution is well
described by a free energy functional which we assume composed of
three terms, namely the {\it exchange} energy ${E}_{\rm ex}$, the
{\it elastic} energy $E_{\rm el}$ and the {\it elastic-magnetic}
energy $E_{\rm em}$. In Section~\ref{model} we detail the three
energetic terms and, after some simplifications, derive the proposed
functional for describing some phenomena. Assuming the hypothesis of
radially symmetric maps, i.e.
$$ \m = ( \cos \theta \sin h(r), \sin \theta \sin h(r), \cos h(r)
),$$ we get to the analysis of a one-dimensional energy functional
that can be expressed in terms of the only scalar function $h$. The
effect of the elastic deformation reveals through a positive
parameter $\mu$ which characterizes the connection between the
magnetic and elastic processes. In Section~\ref{minproblem} 
 the minimisation of the energy functional, namely
$$E(h)= \pi \int_0^1 \left[h_r^2+\left(\frac{\sin h}{r}\right)^2 -
\frac{\mu}{2} (\sin2h )^2 \right] rdr,$$ is the aim of  our paper.
In particular, we prove that there exists a critical value $\mu^0$
such that for $\mu \le \mu^0$ the functional energy is not negative
and there is only a global minimiser that is the trivial solution
$h\equiv 0$; for $\mu
> \mu^0$ other nontrivial minimisers appear, moreover the energy
takes negative values. The local
bifurcation analysis is carried out. More precisely we prove that at
the point $\mu^0$, two branches of  minimisers, with small norm,
bifurcate from the trivial stable solution. This local analysis does
not exclude the existence of other solutions of the minimisation
problem even for $\mu = 0$ (see also the results by Brezis and Coron
in \cite{BC} concerning the solutions of harmonic maps from the unit
disk in $\mathbb{R}^2$ to the sphere $S^2$).

For the modelling of magneto-elastic interactions see also \cite{BPGV}, \cite{Bw},
\cite{CPV}, \cite{csvv}, \cite{csvv1}, \cite{He}, \cite{VV}. Magneto-viscoelastic problems are studied in \cite{GVS2010},  \cite{GVS2012} and  \cite{MGVS2017}.
Moreover we recall that the phenomenon of bifurcation of minimising
harmonic maps has been studied by Bethuel, Brezis, Coleman, H{\rm
\'{e}}lein (see \cite{BBCH}) in a different physical context.

\medskip

\section{The model} \label{model}
\setcounter{equation}{0}

 We start with the general three-dimensional theory.
We assume $\Omega \subset \mathbb{R}^3$ is the volume of the
magneto-elastic material and $\partial \Omega$ its boundary. Let
$x_i,\, i=1,2,3$ be the position of a point $\x$ of $\Omega$ and
denote by
$$u_i=u_i(\x),\qquad i=1,2,3$$
the components of the displacement vector $\u$ and by
$$\varepsilon_{kl}(\u)=\frac{1}{2}(u_{k,l}+u_{l,k}),\qquad k,l=1,2,3$$
the deformation tensor where, as a common praxis, $u_{k,l}$ stands
for ${\frac{\partial u_k}{\partial x_l}}$. Moreover we denote by
$$m_{j}=m_{j}(\x),\qquad j=1,2,3$$
the components of the magnetisation vector $\m$ that we assume of
unit modulus, i.e. $|\m|=1$. \\
In the sequel, where not specified,
the Latin indices vary in the set \{1,2,3\} and the summation over repeated indices is assumed. 
We first define the exchange energy which arises from exchange 
neighbourhood interactions as
\begin{equation} \label{Eex}
E_{\rm ex}(\m) =\frac{1}{2} \int_{\Omega} a_{ij} m_{k,i}
m_{k,j}d\Omega
\end{equation}
where $a_{ijkl} = a_1\delta_{ijkl} + a_2\delta_{ij}\delta_{kl}$ with $a_1,a_2\ge0$ and  $\delta_{ijkl}=\delta_{ik}\delta_{jl}$ is the fourth-order identity tensor.
%where $(a_{ij})$ is a symmetric positive definite matrix which is
%supposed diagonal for most materials with all diagonal elements
%equal to a positive number $a$.
% The magneto-elastic energy for
%cubic crystals is assumed to be
This integral represents the interface energy between magnetised domains with different orientations. For most magnetic materials $div\, {\bf m}=\delta_{ij}{m}_{i,j}=0$, so hereafter we assume $a_1=a>0$ and $a_2=0$ (see \cite{LL}). 
The magneto-elastic energy is due to the coupling between the magnetic moments and the elastic lattice. For cubic crystals it is assumed to be
\begin{equation} \label{Eem}
E_{\rm em}(\m,\u) =\frac{1}{2} \int_{\Omega} \lambda_{i j k l}
m_{i} m_{j} \varepsilon_{kl}(\u)d\Omega
\end{equation}
where  $\big.\mathbb{L}= \{{\lambda}_{klmn}\}$ denotes the \textit{ {magneto-elasticity tensor}} whose entries $ \lambda_1, \lambda_2, \lambda_3\ge0$, and $\lambda_{i j k l}=\lambda_{1} \delta_{i j k l}+ \lambda_{2}
\delta_{i j }\delta_{kl}+ \lambda_{3}(\delta_{i k}\delta_{j
l}+\delta_{i l}\delta_{j k})$ with $\delta_{i j k l}=1$ if
$i=j=k=l$ and $\delta_{i j k l}=0$ otherwise. Moreover we
introduce the elastic energy
\begin{equation} \label{Eel}
E_{\rm el}({\bf u})= \frac12 \int_\Omega \sigma_{ijkl} 
 \varepsilon_{ij}({\bf u})\varepsilon_{kl}({\bf u})\,d\Omega
 \end{equation}
where $\Big.\mathbb{E}= \{ \epsilon_{lm}\}$ indicates the  \textit{ {strain tensor}}  $\sigma_{ijkl}$ %is the elasticity tensor 
satisfying the following symmetry property 
$$\sigma_{ijkl}=\sigma_{klij}=\sigma_{jilk}$$
and moreover the inequality
$$\sigma_{ijkl}\varepsilon_{ij}\varepsilon_{kl}\ge\beta\varepsilon_{ij}\varepsilon_{ij}$$
holds for some $\beta>0$. In the isotropic case
$$\sigma_{ijkl} = \tau_1\delta_{ijkl} + \tau_2\delta_{ij}\delta_{kl},\qquad \tau_1, \tau_2\ge0.$$
The resulting energy functional ${E}$ is given by
\begin{equation}
E(\m,\u)= E(m,u) = E_{ex} (m) + E_{em} (m,u) + E_{ve} (u),
\end{equation}
which {{after some manipulations \cite{BPGV, GVS2012}}}, under the assumption the material is isotropic, reads
\begin{multline}
E(\m,\u)= 
 \frac{1}{2}\int_{\Omega} a|\nabla {\m}|^2 d\Omega +
\frac{1}{2}\int_{\Omega} \left[\tau_1
|\nabla {\u}|^2+\tau_2 (div\, {\u})^2\right] d\Omega +\\
+\frac{1}{2}\int_{\Omega} \left[\lambda_{1}\delta_{k l i j}% {\partial\over \partial_i} 
u_{j,i}
m_{k}m_{l} +\lambda_{2} |{\m}|^2 div \,{\u }+2\lambda_3(\nabla
u_i\cdot {\m}) m_{i} \right] d\Omega~.
\end{multline}
\subsection{A simplified 2D model}

To get the proposed model we make some approximations. First of all
we assume
 $\Omega \subset \mathbb{R}^2$
 and neglect the components in plane of the
 displacement vector $\u$, i.e.  we assume  ${\u}=(0,0,w)$, which implies $div\, u=0$ since $w$ depends only on the plane coordinates. Let $\lambda_3= \lambda$ be a positive constant,
 %we assume  ${\u}=(0,0,w)$ and
setting\footnote{No need to prescribe  nor $ \lambda_2$ nor $\tau_2\ge0$ since they both appear only as  factors of $div\, u$;  also $\lambda_1$ can be left arbitrary; indeed,   $\delta_{klij}u_{j,i}=0$ since $u_{j,i}\neq0$ only if $j=3$ and $i=1,2$ but $\delta_{klij}=0$ when $ j\neq i$.} %$ \lambda_1=\lambda_2=0,   \lambda_3 =\lambda$, 
$\tau_1= 1$  and $a=1$,  
the functional
${E}$ reduces to
\begin{equation} \label{E}
\displaystyle{{E}(\m,w)= \frac{1}{2}\int_{\Omega} \left( |\nabla {\m}|^2
+ 2\lambda m_3 (m_{\alpha} w_{,\alpha})+|\nabla w|^2\right)d\Omega}
\end{equation}
where the Greek indices vary in the set \{1,2\}.

Setting $\Omega \equiv D = \{ (x,y) \in \mathbb{R}^2 : \, x^2+y^2 <
1\}$ and assuming radial symmetry, further to $w=w(r)$, we can express the components of the  vector $\m$ 
in terms of $r$, that is, of the form
$$ \m = \left( \frac{x}{r} \sin h(r), \frac{y}{r} \sin h(r), \cos h(r)
\right), \qquad r=\sqrt{x^2+y^2},$$ 
where $h: (0,1)\subset \R \to \R$ is an unknown regular function.
Using the fact that $\partial_x r = \frac{x}{r}$ and $\partial_y r = \frac{y}{r}$ we deduce by the chain rule, where  $h_r:=\displaystyle{{ d h} \over {dr}}$  denotes the derivatives of the $h$ with respect to  the variable $r$, it follows
%  noting by $'$ the derivative in $r$
\begin{equation*}
\partial_x \m = \left(\frac{\sin h}{r}+ \frac{x^2}{r}\left(\frac{\sin h}{r}\right)_{\!r}, ~\frac{xy}{r}\left(\frac{\sin h}{r}\right)_{\!r} , ~\frac{x}{r}\left(\cos h\right)_{\!r}\right)
\end{equation*}
\begin{equation*}
\partial_y \m = \left({\frac{xy}{r}\left(\frac{\sin h}{r}\right)_{\!r} ,~ \frac{\sin h}{r}+ \frac{y^2}{r}\left(\frac{\sin h}{r}\right)_{\!r} , ~\frac{y}{r}\left(\cos h\right)_{r} }\right).
\end{equation*}
Thus we get 
\begin{equation*}\displaystyle
\begin{aligned}
|\nabla \m |^2&= \left[\frac{\sin h}{r}+ \frac{x^2}{r}\left(\frac{\sin h}{r}\right)_{\!r} \right]^2 +\left[{\frac{\sin h}{r}+ \frac{y^2}{r}\left(\frac{\sin h}{r}\right)_{\!r} }\right]^2 +
2\left[{\frac{xy}{r}\left(\frac{\sin h}{r}\right)_{\!r}}\right]^2 +[\left(\cos h\right)_{\!r} ]^2\\
&=2\left(\frac{\sin h}{r}\right)^2 + 
\frac{x^4+2x^2y^2 +y^4}{r^2}\left[{\left(\frac{\sin h}{r}\right)_{\!r}}\right]^2 +2 \frac{x^2+y^2}{r^2}\sin h \left(\frac{\sin h}{r}\right)_{\!r} + h_{r}^2 (\sin h)^2 \\
&= 2\left(\frac{\sin h}{r}\right)^2 + {r^2}\left[\left(\frac{\sin h}{r}\right)_{\!r}\right]^2 +2 \sin h\left(\frac{\sin h}{r}\right)_{\!r} + h_{r}^2(\sin h)^2 \\
&= \left(\frac{\sin h}{r}\right)^2 + \left[\frac{\sin h}{r}+{r}\left(\frac{\sin h}{r}\right)_{\!r}\right]^2
 +h_{r}^2 (\sin h)^2 \\
&= \left(\frac{\sin h}{r}\right)^2 + \left[\frac{\sin h}{r}+{r}\left(\frac{r h_{\!r} \cos h - 
\sin h}{r^2}\right)\right]^2 + h_{r}^2 (\sin h)^2 \\
&= \left(\frac{\sin h}{r}\right)^2  +h_r^2.
\end{aligned}
\end{equation*}`

So the energy (\ref{E}), when we recall the assumed radial symmetry implies also $w=w(r)$, adopting the notation $w_r:=\displaystyle{{ d w} \over {dr}}$, becomes
$$
E(h,w)= \pi \int_0^1 \left[h_r^2 +\left(\frac{\sin h}{r}\right)^2
+\lambda \sin2h \, w_r + w_r^2\right] rdr$$
 and from that we deduce the governing
equations
\begin{equation} \label{eqs}
\left \{ \begin {array}{l}\displaystyle{ h_{rr} +\frac{h_r}{r}
-\frac{\sin2h}{2r^2} -\lambda \cos2h \,w_r =0}\\
\\
\displaystyle{ w_{rr} + \frac{w_r}{r} + \frac{\lambda}{2} \left[
(\sin 2h)_r + \frac{\sin 2h}{r}\right]=0}.
\end{array}\right.
\end{equation}
We prescribe the following boundary conditions
\begin{equation} \label{bcw}
w_r(0)=0,\,\,w(1)=0,
\end{equation}
where the first condition is motivated by the symmetry assumptions, while the second one corresponds 
to prescribe  the boundary of $\Omega$ is fixed, and %pectively, 
\begin{equation} \label{bch}
 h_r(1)=0.
 \end{equation}
 Solving the second equation of (\ref{eqs}) which can be written
 \begin{equation*}
(r w_r)_{r} + \frac{\lambda}{2}(r \sin 2h)_{r} =0 \Leftrightarrow w_r = -\frac{\lambda}{2}\sin 2h
\end{equation*}
where the double implication is guaranteed when we set $h(0)=0$. Then, letting
$\mu= \lambda^2/2$ we get  the equation
\begin{equation} \label{eqh}
h_{rr} +\frac{h_r}{r} -\frac{\sin 2h}{2r^2} +\mu \sin 2h \cos2h  =0
\end{equation}
and the energy E becomes
\begin{equation} \label{Eh}
E(h)= \pi \int_0^1 \left[h_r^2+\left(\frac{\sin h}{r}\right)^2 -
\frac{\mu}{2} (\sin2h )^2 \right] rdr.
\end{equation}

The variational analysis of the functional $E(h)$ is the objective
of the following section.

\section{The minimisation problem} \label{minproblem} 

\begin{lemma} \label{lemma1}
 Let us define
\begin{equation}
V=\{v ~\vert ~ v_r, {v\over r}\in L^2(0,1;r dr)\}.
\end{equation}
$V$ is a Hilbert space equipped with the norm
\begin{equation}
\displaystyle{\vert\!\vert v \vert\!\vert^2=
\int_0^1{ ({ v_r}^2+ {v^2\over r^2}) r dr~.}}
\end{equation}
\end{lemma}

\noindent\begin{proof}

Let $v_n$ be a Cauchy sequence in $V$, $\{(v_n)_r\}, \displaystyle{\left\{{v_n\over r}\right\}}$ are Cauchy sequences in 
 $$L^2(r dr)=L^2(0,1;r dr)$$ and there exist $h, g$ such that  
\begin{equation}\label{m1}
\{(v_n)_r\}, \displaystyle{\left\{{v_n\over r}\right\}} \to g,h~~\text{in}~~L^2(r dr)
\end{equation}
Set $\tilde h=h r$. Since
\begin{equation}\label{1m}
\int_0^1 \left({v_n\over r}-h\right)^2  r dr~\to 0
\end{equation}
one has
\begin{equation}
{v_n\to rh ~~\text{in}~~L^2 \big(0,1;{dr\over r}\big)~.}
\end{equation}
but also in ${\cal D}^{\prime}(0,1)$ so that
\begin{equation}
{(v_n)_r\to \tilde h_r ~~\text{in}~~{\cal D}^{\prime}~.}
\end{equation}
We deduce from \eqref{1m} that $\tilde h_r=g$ and thus $ \tilde h\in V$ and since
\begin{equation}
(v_n)_r, {v_n\over r} \to \tilde h_r, {\widetilde h\over r}~~\text{in}~~L^2(r dr)
\end{equation}
one has $v_n\to \tilde h \in V$. This completes the proof of  Lemma \ref{lemma1}. 

\end{proof}\hfill\eproof

\begin{lemma} \label{lemma2}

$$ V \subset \{v\in C([0,1]) ~\vert ~v(0)=0 \} $$
\end{lemma}

\noindent\begin{proof}

For $x,y\in (0,1]$ one has 

\begin{equation}
\vert x v(x)- y v(y) \vert = \left\vert \int_x^y {(r v)_r dr }\right\vert =  \left\vert \int_x^y { r (v_r +{v\over r}) dr }\right\vert 
\\ \le \int_x^y { r (|v_r| +{{|v|}\over r}) dr ~.}
\end{equation}
Using the Cauchy-Young  inequality $a \le {1\over 2} a^2 + {1\over 2}$ one gets

\begin{equation}
\vert x v(x)- y v(y) \vert \le \int_x^y \left\{ { {r\over 2} \left({v_r}^2 +{{v^2}\over {r^2}}\right) }
+ r \right\}dr \to 0 ~~
\text{when}~~ y \to x.
\end{equation}
It follows that $v$ is continuous at any point where $r\neq0$ on $[0,1]$. 
Now, one has also
\begin{equation}
\begin{array}{cl@{\hspace{0.5ex}}c@{\hspace{1.0ex}}l}   
\displaystyle v(x)^2 -  v(y)^2 =\!\!  \int_x^y \!\!  { {d\over {dr}} {v(r)}^2 } dr = \!\!  \int_x^y \!\!  { 2 {v_r \, v} }~ dr =
\!\! 
 \int_x^y \!\!  {2 \sqrt{r} ~v_r {v\over {\sqrt{r}}} }dr \\ \\
 \displaystyle \le \int_x^y\!\!  \left\{ { {r} ({v_r}^2 +{{v^2}\over {r^2}})} \right\} ~dr \le \varepsilon 
\end{array}
\end{equation}
for $x,y$ small enough (we used again  the Cauchy-Young  inequality).
\smallskip
Thus, when $x\to0$, $v(x)^2$ is a Cauchy sequence and there exist  $l\ge 0$ such that
\begin{equation}
\displaystyle{\lim_{x\to 0} {v(x)^2}=l}.
\end{equation}
If $l>0$ one has for $\varepsilon$ small enough
\begin{equation}
\Vert v \Vert^2\ge \int_0^1{ {v^2\over r} dr }\ge 
\int_{\varepsilon^2}^\varepsilon{ \left({l\over 2}\right)^2{{dr}\over r} }
={l^2\over 4}(\ln \varepsilon- 2 \ln \varepsilon) = - {l^2\over 4} \ln \varepsilon
\end{equation}
and a contradiction when $\varepsilon\to 0$. Thus, $l=0$ and this completes the proof of the Lemma \ref{lemma2}.
\end{proof} \eproof 

\bigskip
\noindent{\bf Remark}
%One has, of course, 
Since $V\subset H^1( \varepsilon,1)$, it follows that $V \subset C^{1/2}( \varepsilon,1)$ for every $\varepsilon$.

One sets 

\begin{equation}
E(h)= \pi \int_0^1 \left\{ h_r^2+  ({{\sin h}\over{r}})^2 - \frac{\mu}{2} (\sin 2h)^2 \right\} rdr.
\end{equation}

\bigskip\noindent
One would like to show that $E(h)$ possesses a minimiser on $V$ for any $\mu$.

\noindent\begin{lemma} \label{lemma3}

The energy $E(h)$ is bounded from below on $V$ and one can find a minimising sequence $v_n$ such that
\begin{equation}\label{2m}
0\le v_n\le {\pi \over 2}~.
\end{equation}
\end{lemma}

\noindent\begin{proof}

One has clearly for every $h\in V$
\begin{equation}
E(h)\ge -\pi {|\mu|\over 2}\int_0^1 r dr =  -\pi {|\mu|\over 4}~.
\end{equation}
Thus $$ I= \inf_{h\in V} E(h) $$
exists. Let us denote by $v_n$ a sequence such that
$$ E(v_n) \to I~. $$

If $v_n\in V$, then also   $\vert v_n \vert\in V$ and one has $$ E(v_n)=  E(|v_n|)$$
so, without loss of generality, we assume $v_n\ge 0$.

\bigskip
\begin{figure}[H]
\centering
\epsfig{file=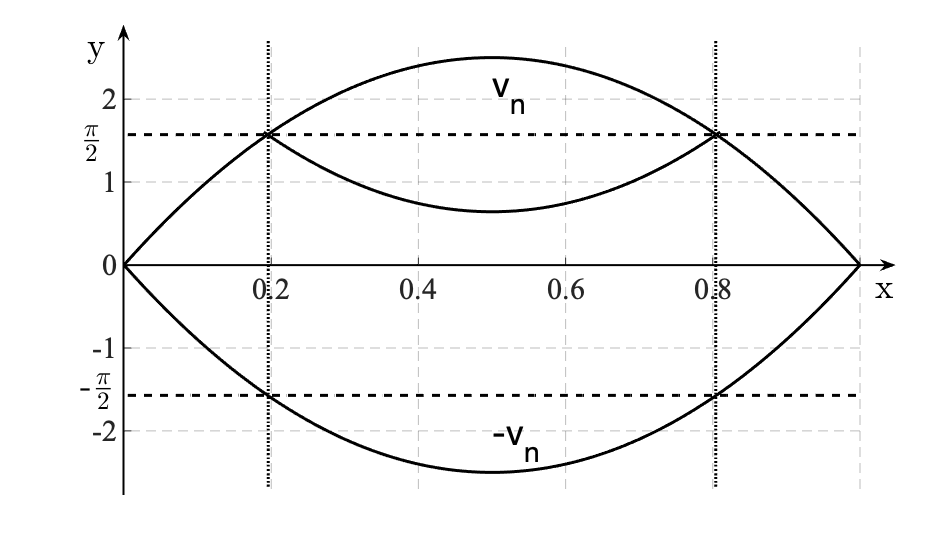, scale=0.31} 
\caption{graphical representation.}
\label{fig4}
\end{figure}

\noindent
Then on $v_n> {\pi \over 2}$, we replace $v_n$ by $-v_n +\pi$ (cfr. Fig. 1).
It is clear that
\begin{equation}
\tilde v_n=v_n X_{\{v_n\le {\pi \over 2}\}} + (-v_n +\pi)X_{\{v_n> {\pi \over 2}\}}
\end{equation}
satisfies $\tilde v_n\in V$ and $$E(\tilde v_n)=E(v_n).$$
This completes the proof of the Lemma \ref{lemma3}.
\end{proof}

\hfill\eproof

\begin{remark}\label{rem-p4} It could be that $-v_n +\pi$ achieves negative values, but clearly, after a finite number of operations like the one we just did we get a $v_n$ satisfying \eqref{2m}.
\end{remark}

\begin{lemma} \label{Lemma 4} There exists a minimiser $\tilde h$ of $E$ in $V$ satisfying
\begin{equation}
0 \le \tilde h\le  {\pi \over 2}~.
\end{equation}
\end{lemma}

\noindent\begin{proof}  
We consider the sequence $\{v_n\}$ constructed in Lemma \ref{lemma3}.
We claim that $\{v_n\}$  is bounded in $V$ independently of $n$. Indeed, one has, since
for some constant $\lambda>0$ one has $\displaystyle{\left({{\sin x}\over{x}}\right)^2\ge \lambda}$,
$\forall x\in[0,{\pi \over 2}]$,
\begin{equation}  \label{m3}
\begin{array}{cl@{\hspace{0.5ex}}c@{\hspace{1.0ex}}l}
\displaystyle  \int_0^1 {r \left\{ {({v_n})_r}^2 +{{{v_n}^2}\over {r^2}} \right\} dr \le 
\int_0^1 {r \left\{  {({v_n})_r}^2  +{1\over \lambda}\left({{\sin v_n}\over {r}}\right)^2 \right\} dr }}\\
\displaystyle \le \left(1 \vee {1\over \lambda}\right) \int_0^1 {r \left\{  {({v_n})_r}^2  + 
\left({{\sin v_n}\over {r}}\right)^2 \right\} dr \le C}
\end{array}
\end{equation}
where $C$ is a constant independent of $n$ and $\vee$ denotes the maximum of two numbers. 
Recall that since ${v_n}$ is a minimising sequence one has, for $n$ large enough, 
$$ E({v_n}) \le E(0) =0$$
i.e. see the definition of $E$
\begin{equation}
\pi \int_0^1 {\left\{ {({v_n})_r}^2 + \left({{\sin v_n}\over {r}}\right)^2 \right\} r dr}\le 
\pi {|\mu|\over 2}\int_0^1{{\sin^2 (2 v_n)} r dr} \le \pi {|\mu|\over 4}~.
\end{equation}
Since  $\displaystyle{\{({v_n})_r}\}, \{{{v_n}\over r}\}$ are bounded in $L^2(r dr)$ one finds a subsequence, still labelled by $n$, such that 
$$  {{v_n}\over r} {{\rightharpoonup h}} ~~, ~~ {({v_n})_r} \rightharpoonup  g ~ \text{in} ~  L^2(r dr)~.$$
Set $\tilde h =h r$. The first weak convergence above reads
\begin{equation*}
\displaystyle \int_0^1 {{v_n}\over r} \Psi r dr \to  \int_0^1 h \Psi r dr~~~~, ~~~\forall  \Psi\in L^2 (rdr).
\end{equation*}
In particular, taking $\Psi \in {\cal D}(0,1)$ one see that
$$ {v_n} \to \tilde h =h r  ~ \text{in} ~ {\cal D}^{\prime}(0,1)$$
and thus, by the continuity of the derivative in ${\cal D}^{\prime}$
$$({v_n})_r \to \tilde h_r  = g ~~\text{in}~~  {\cal D}^{\prime}(0,1)~.$$
Thus, we have $ \tilde h\in V$. For any $k\ge 2$ one has also, thank to \eqref{m3}, that 
$ {v_n}$ is bounded in $\displaystyle H^{1}\left({1\over k},1\right)$. Thus, by induction, one can find a subsequence
$\{n_{k}\}$ extracted from $\{n_{k-1}\}$ such that
$$ {v_{n_{k}} }\to \tilde h   ~ \text{in} ~ L^{2}\left({1\over k},1\right)~ \text{and a. e..}$$
Then clearly
$$ v_{n_{k}} \to \tilde h  ~~ ~  \text{a.e. on} ~ (0,1).$$
By the dominated Lebesgue theorem one has then that
$$ r \sin 2 {v_{n_{k}} } \to  r \sin 2  \tilde h ~~ ~ \text{in} ~ L^{2}(0,1) $$
$$  \sin {v_{n_{k}} } \to   \sin    \tilde h ~~ ~ \text{a.e.~ on} ~ (0,1)~. $$
Then, since $ x\mapsto x^2$ is convex by the Fatou lemma one has 

\begin{equation}\label{weak-eps} 
\begin{array}{cl@{\hspace{0.5ex}}c@{\hspace{1.0ex}}l}   
{\displaystyle{ I={\underline{\lim}} ~E({v_{n_k} }) = 
\pi\, {\underline{\lim}} \int_0^1 {\left\{ {({v_{n_k} })_r}^2 + \left({{\sin {v_{n_k} }}\over {r}}\right)^2 \right\} r dr}-
\pi \int_0^1{{\mu\over 2}({\sin (2 {v_{n_k} })})^2 r dr}\ge}}\\ \\
\displaystyle{\ge \pi ~ {\underline{\lim}} \int_0^1 {{({v_{n_n} })_r}^2 r dr}+  \pi ~{\underline{\lim}} \int_0^1 {\left({{\sin {v_{n_k} }}\over {r}}\right)^2 r dr}-
{\pi\mu\over 2} \int_0^1{({\sin (2 \tilde h)})^2 r dr}\ge}\\ \\
{\displaystyle{\ge \int_0^1 {({\tilde h_r})^2 r dr} +  \pi  \int_0^1 {{\underline{\lim}}{\left({{\sin {v_n}}\over {r}}\right)^2} r dr} -
{\pi\mu\over 2}\int_0^1{{(\sin (2 \tilde h))^2} r }dr} =}\\ \\
\displaystyle{\hphantom{1\over1}=E({\tilde h}) = I}~.
\end{array}
\end{equation}
This shows that ${\tilde h}$ is the minimiser that we are looking for.
\end{proof} \eproof  
 
 \begin{lemma}\label{Euler-eq}{
The Euler equation of the minimising problem is given by
\begin{equation} \label{m1b}
\left \{ \begin {array}{l}\displaystyle{ -h_{rr}
-\frac{h_r}{r}
+\frac{\sin 2h}{r^2} = \mu \sin 2h \cos 2h } ~~~\text{in}~ (0,1)\\
\\
h(0) =h_r(1)=0
\end{array}\right.
\end{equation}}
\end{lemma}

\noindent\begin{proof}
If $h$ is a minimiser of $E$ on $V$ one has
$$ {d \over{d \lambda}} E(h+\lambda v) \vert_0=0~~~,~~ \forall v\in V $$
Since 
\begin{equation}
\displaystyle{ E(h+\lambda v)=
\pi   \int_0^1 {\left\{ (h+\lambda v)_r^2 +{ {{\sin(h+\lambda v)^2}}\over{{r}^2}} 
-{\mu\over 2} {\sin (2 (h+\lambda v))}^2
\right\} }r dr}~.
\end{equation}
One gets $\forall v$
\begin{equation}\label{m2b} 
\begin{array}{cl@{\hspace{0.5ex}}c@{\hspace{1.0ex}}l}   
\displaystyle{ 
 \int_0^1 {\left\{ 2 h_r v_r + 2 {{ \sin h\cos h }\over {{r}^2}}v 
-2{\mu}\, {\sin (2 h)\cos (2 h)}\, v
\right\} r dr}=0 }\\ \\
\displaystyle{ \Longleftrightarrow
 \int_0^1 {\left\{  h_r v_r +  {{\sin  (2 h)}\over {2{r}^2}} v
-{\mu}\, {\sin (2 h)\cos (2 h)}\, v
\right\} r dr}=0~~~, ~~~\forall v\in V}~.
\end{array}
\end{equation}
Thus, in the distributional sense
\begin{equation}\label{m3b}
\begin{array}{cl@{\hspace{0.5ex}}c@{\hspace{1.0ex}}l}   
\displaystyle{- ( r h_r )_r +  {{\sin  (2 h)}\over {2{r}}} 
-{\mu}\,r\, {\sin (2 h)\cos (2 h)} =0} \\ \\
\displaystyle{ \Longrightarrow  - r h_{rr} -h_r +  {{\sin  (2 h)}\over {2{r}}} 
-{\mu}\,r\, {\sin (2 h)\cos (2 h)} =0~.}
\end{array}
\end{equation}
Dividing by $r$ we get the first equation of \eqref{m1b}. Integrating by parts in \eqref{m2b}
and using \eqref{m3b} we get
\begin{equation*}
\displaystyle{  \int_0^1 { ( r h_r v )_r - ( r h_r)_r v+  {{\sin  2 h v}\over {2{r}}} 
-{\mu} {\sin (2 h)\cos 2 h r v}} =0~~~, ~~\forall v \in V } 
\end{equation*}
i.e.
\begin{equation*}
\displaystyle{  \int_0^1 { ( r h_r v )_r } =0~~~, ~~\forall v \in V } 
\end{equation*}
which gives 
$$ h_r(1)=0~. $$
(in a weak sense) $h(0)=0$ follows from $h \in V$. This completes the proof of the Lemma.

 \eproof  \end{proof}

 \begin{lemma}\label{p2b}
 If $h\neq 0$ is a nonnegative minimiser of $E$ on $V$ then $h>0$ on $(0,1)$.
\end{lemma}
\begin{proof}
Indeed, if $h$ vanishes at $r_0\in(0,1)$ then, since $h$ is smooth and $r_0$ is a minimum
for $h$, one would have
$$ h(r_0)=  h_r(r_0)=0$$
then from the theory of o.d.e's (see [1]), $h\equiv 0.$ 
\eproof  \end{proof}

 \begin{lemma}\label{4.8}
 If $h$ is a positive minimiser of $E$  then $0<h\le {\pi\over 2}$.
\end{lemma}
\begin{proof}
If not then $h$ constructed as in the figure before (Fig.1) is a minimiser but it has a jump in the derivative unless this one is $0$. But then $h= {\pi\over 2}$ is solution of the o.d.e. on $h >{\pi\over 2}$  and a contradiction follows. Note that the solution of the elliptic equation (\ref{m1}) is smooth on $(0,1)$. 
\end{proof}

 \begin{lemma}\label{4.9b}
A minimiser cannot vanish  on $(0,1)$ unless it vanishes identically.
\end{lemma}
\begin{proof}
If  $h$ is  a minimiser,  $\vert h \vert $ is also a minimiser. But, then, $\vert h \vert $ would 
have a jump discontinuity in its derivative unless when it vanishes so does $h_r$. 
This implies (theory of o.d.e's),  $h=0$.
\end{proof}

\rightline \eproof

\begin{lemma}\label{4.9}
If $h \in V$ then   $\sin (kh) \in V ~~, ~~ \forall k\in\R$.
\end{lemma}
\begin{proof}
One has
\begin{equation}
\sin (kh)_r= k h_r\cos (kh)~~,~~\vert \sin (kh)\vert \le \vert kh \vert.
\end{equation}
Therefore one has 
\begin{equation}
\begin{array}{cl@{\hspace{0.5ex}}c@{\hspace{1.0ex}}l}   
\displaystyle
\Vert \sin (kh) \Vert^2 =  \int_0^1 \left\{  \sin (kh)_r^2 + \left( { \sin (kh) \over r}\right)^2 \right\} r dr\\ \\
\displaystyle{ \le  \int_0^1 \left\{ k^2 \cos (kh)^2  h_r^2+ k^2 {h^2 \over {r^2}} \right\} r dr \le  k^2 \Vert h\Vert ^2 }
\end{array}
\end{equation}
\end{proof}
\bigskip

\noindent It easy to check that $h\equiv 0$ solves (\ref{bch}),
(\ref{eqh}) and hence it is a stationary point of the functional
(\ref{Eh}). \\
Let $\gamma_0$ be the first eigenvalue of the problem
\begin{equation} \label{eigen}
\left \{ \begin {array}{l}\displaystyle{ -\phi_{rr}
-\frac{\phi_r}{r}
+\frac{\phi}{r^2} = \gamma \phi }\\
\\
\phi(0) =0, \quad \phi_r(1)=0 ~.
\end{array}\right.
\end{equation}
\begin{lemma}\label{gamma-0}
$$\gamma_0>1~~.$$
\end{lemma}
\begin{proof}
Suppose not,  i.e. $\gamma_0\le1$.
Let $\phi$ be the corresponding positive (or nonnegative) eigenfunction.
One has 
\begin{equation}
 \begin {array}{l} \displaystyle{ -\phi_{rr}
-\frac{\phi_r}{r}= \phi (\gamma_0- {1\over {r^2}}) \le 0
~~~\text{since}~r\in (0,1)}\\ \\
\displaystyle{ (r \phi_r)_r \ge 0 \Longrightarrow r \phi_r \nearrow ~~
\Longrightarrow ~~r \phi_r \le 0 ~~\text{since}~~   \phi_r (1)=0~.}
\end{array}
\end{equation}
Thus, the maximum of $\phi$ is achieved at $0$ but, since $\phi(0)=0$,
we get a contradiction i.e. $\phi\equiv0$.
\hfill\eproof
\end{proof}

\medskip\noindent
We have the following bifurcation lemma.

\begin{lemma}\label{biforcation}
If $\mu \le \gamma_0/2 $ we have $E(h) \ge 0$ and the global minimum
is attained  only for $h\equiv 0$. For $\mu > \gamma_0/2 $ the
global minimum is negative.
\end{lemma}
\begin{proof} The first equation of \eqref{eigen}  can also be written after a multiplication by $r$ as 
\begin{equation*}
\displaystyle{ -(r\phi_r)_{r}+\frac{\phi}{r} = \gamma \phi r}.
\end{equation*}
Multiplying by $\phi$ and integrating over $(0,1)$ we derive by definition of $\gamma_0$ that 
\begin{equation}\label{pointca}
\int_0^1 \left(\phi_r^2 + \frac{\phi^2}{r^2} \right)rdr \geq \gamma_0 \int_0^1  \phi^2 rdr ~~\forall \phi ~\text{ with }~\phi(0) =0, \quad \phi_r(1)=0.
\end{equation}

We divide the proof in two parts:

\bigskip
\noindent ({\it i} ) \quad $\mu \le \gamma_0/2$ \\

In this case we have (using \eqref{pointca} with $ \phi = \sin h$)
\begin{equation*}
\begin{aligned}
E(h) &=\pi \int_0^1
\left[(\cos h)^2 h_r^2+\left(\frac{\sin h}{r}\right)^2 - 2\mu (\sin h )^2(\cos h)^2 + (1-(\cos h)^2) h_r^2 \right] rdr \\
&\geq \int_0^1
\left[\gamma_0 (\sin h )^2 - 2\mu (\sin h )^2(\cos h)^2 + (1-(\cos h)^2) h_r^2 \right] rdr \\
& = \int_0^1
\left[ (\gamma_0  - 2\mu ))(\sin h )^2 + (1-(\cos h)^2) (2\mu (\sin h )^2 + h_r^2) \right] rdr \\
&\geq 0 =E(0)
\end{aligned}
\end{equation*}
the equality taking place only for $h=0$.

\bigskip
\noindent({\it ii }) \quad $\mu > \gamma_0/2$ \\

Let us denote by $\phi_0$ the first positive normalised eigenfuntion to \eqref{eigen}.

One has for $\epsilon >0$
\begin{equation*}
\begin{aligned}
E(\epsilon \phi_0) &\leq \pi \int_0^1
\left[(\epsilon \phi_0)_r^2+\frac{(\epsilon \phi_0)^2}{r^2} - \frac{\mu}{2} (\sin (2\epsilon \phi_0))^2\right] rdr \\
&= \pi \int_0^1
\left[\gamma_0(\epsilon \phi_0)^2 - \frac{\mu}{2} (\sin (2\epsilon \phi_0))^2\right] rdr. \\
\end{aligned}
\end{equation*}
Using with $x= 2\epsilon \phi_0$ the formula 
\begin{equation*}
\sin x = x - \int_0^1 (1 - \cos (tx)) xdt 
\end{equation*}
$E(\epsilon \phi_0) $ can be written as %it comes 
\begin{equation*}
\begin{aligned}
E(\epsilon \phi_0) 
&= \pi \int_0^1
\left[(\epsilon \phi_0)^2 \{\gamma_0- 2\mu (1-\int_0^1 (1 - \cos (2t\epsilon \phi_0))dt )^2\}\right] rdr \\
&< 0 =E(0)
\end{aligned}
\end{equation*}
for $\epsilon$ small since 
\begin{equation*}
\int_0^1 (1 - \cos (2t\epsilon \phi_0))dt \to 0 
\end{equation*}
when $\epsilon \to 0$.
 \eproof  \end{proof}

\bigskip

\noindent
{\bf Alternative proof of $(ii)$
}

\medskip\noindent
Suppose $h\neq 0$ is a minimiser of $E$ one has
\begin{equation}
E(h)< E(0)
\end{equation}
i.e. 
\begin{equation}
\begin{array}{cl@{\hspace{0.5ex}}c@{\hspace{1.0ex}}l}  
\displaystyle \int_0^1
\left\{  h_r^2+\left(\frac{\sin h}{r}\right)^2 \right\} rdr < {\mu \over 2} \int_0^1
 \left({\sin 2 h}\right)^2  rdr = 2\mu \int_0^1 {\sin  h}^2 {\cos  h}^2 rdr  \\ 
\\
\displaystyle \Longrightarrow~~  \int_0^1
\left\{  {\cos  h}^2 h_r^2+\left(\lambda\frac{ \sin h}{r}\right)^2 \right\}rdr <\gamma_0 
 \int_0^1 {\sin  h}^2  rdr \\ \\
\displaystyle \Longrightarrow~~ \gamma_0 >
{ \int_0^1\left\{   {\sin  h}_r^2 +\left(\frac{\sin h}{r}\right)^2 \right\}  rdr \over
\int_0^1 {\sin  h}^2  rdr }
\end{array}
\end{equation}
and a contradiction since ${\sin  h}\in V$ with the definition of $\gamma_0$.

\medskip\noindent
Consider the problem
\begin{equation} \label{one}
\left \{ \begin {array}{l}\displaystyle{ -h_{rr}
-\frac{h_r}{r}
+\frac{\sin 2h}{r^2} = \mu \sin 2h \cos 2h } ~~~\text{on}~ (0,1)\\
\\
h(0) =h_r(1)=0
\end{array}\right.
\end{equation}

\noindent
\begin{lemma}
If $\mu\le \gamma_0/2$ the only solution of (\ref{one})
such that $ \displaystyle{h\in \left[ -{\pi\over 2}, {\pi\over 2}\right]}$ is $h\equiv0$.
\end{lemma}

\noindent
\begin{proof}
Recall that for any $\phi \in V$ one has by definition of  $\gamma_0$
\begin{equation}\label{due}
\displaystyle{\gamma_0 \int_0^1 \phi^2 r dr \le 
 \int_0^1 \left(\phi_r^2 + {\phi^2\over{r^2}}\right) r dr } ~~.
\end{equation}
Let us write the %first 
equation  (\ref{one}) as 
\begin{equation}
({r h_r})_{r} +\frac{\sin 2h}{2 r} = \mu \sin 2h \cos 2h \, r~.
\end{equation}
Multiply both sides by $\sin 2h$ and integrate on $(0,1)$. It comes
\begin{equation}
\displaystyle{\int_0^1 r \left\{ h_r  ({\sin 2h})_{r} +\frac{(\sin 2h)^2}{2 r^2} \right\} dr= 
\mu \int_0^1 (\sin 2h)^2 \cos 2h ~ r dr} ~.
\end{equation}
One has 
\begin{equation}
\displaystyle{({\sin 2h})_{r} =2 \cos 2h ~ h_r \Longleftrightarrow
h_r= {({\sin 2h})_{r} \over {2 \cos 2h}}}~.
\end{equation}
Thus, the equation above becomes 
\begin{equation}
\displaystyle{\int_0^1 r \left\{   ({\sin 2h})^2_{r} {1\over {\cos 2h}} + \frac{(\sin 2h)^2}{ r^2} \right\} dr= 
2 \mu \int_0^1 (\sin 2h)^2 \cos 2h r dr} ~.
\end{equation}
Suppose that $h$ is such that  $ \displaystyle{h\in \left[ -{\pi\over 2}, {\pi\over 2}\right]}$
 then since $-1\le \cos 2h\le 1$ one gets 
 \begin{equation}
\displaystyle{\int_0^1 r \left\{   ({\sin 2h})^2_{r}  + \frac{(\sin 2h)^2}{ r^2} \right\} dr<
\gamma_0 \int_0^1 (\sin 2h)^2  r dr} ~
\end{equation}
i.e $\sin 2h \in V$ and satisfies an inequality contradicting (\ref{due}),
except if $h\equiv 0$. %\vskip 2cm
\end{proof} \hfill\eproof
%\lblu

\bigskip
\noindent Each minimiser of $E(h)$  solves the problem (\ref{bch}),
(\ref{eqh}). For the solutions of this problem  we can give the
following existence result around the bifurcation point.

\begin{lemma} \label{existence}
There exist two positive numbers $\rho_0$ and $\delta_0$ such that,
the problem (\ref{bch}), (\ref{eqh}) does not have non-zero
solutions for $\mu \in (\gamma_0/2 - \delta_0, \gamma_0/2]$ and
$\|h\|_0 \le \rho_0$. The problem has exactly two solutions $h_1$
and $h_2 = - h_1$ in the sphere $\|h\|_0 \le \rho_0$ for $\mu \in
(\gamma_0/2, \gamma_0/2 + \delta_0)$.
\end{lemma}
\begin{proof}
The proof follows from \cite[Theorem 6.12]{K}. Indeed the equation
(\ref{eqh})  can be written in the form
\begin{equation}\label{eqK}
 2\mu h =  L(h, r) + C(h, r, \mu) + D(h, r, \mu)
\end{equation}
 where $L$ is the linear operator
 $$L(h,r)=-h_{rr} -\frac{h_r}{r} +\frac{h}{r^2}$$
 and $C$, $D$  are given by
 $$C(h, r, \mu)=-\frac{2}{3}\frac{h^3}{r^2}+\frac{16}{3} \mu {h^3},$$
$$D(h, r, \mu)=-(\frac{2h}{2r^2} -\frac{\sin2h}{2r^2}) + \frac{2}{3}\frac{h^3}{r^2}
+\frac{\mu }{2}( 4h - \sin 4h) - \frac{16}{3} \mu {h^3}.
$$

\noindent It is easy to check that
\begin{equation} \label{omo}
C(th,r,\mu) = t^3 C(h,r,\mu),\quad (-\infty < t < \infty)
\end{equation}
and
\begin{equation}\label{D}
\|D(h,r,\mu)\|_0  = o (\| h\|^3).\end{equation} Moreover we
have
\begin{equation}\label{CF}
\displaystyle{((C(\phi^0,r,\mu),\phi^0))_0 =
\int_0^1 \left[-\frac{2}{3}\frac{(\phi^0)^4}{r}+\frac{16}{3} \mu
{(\phi^0)^4} r\right] dr}
> 0,\quad {\rm for}\,\, \mu\ge \frac{\gamma_0}{8}.
\end{equation}
Indeed from $((L(\phi^0,r)-\gamma_0 \phi^0, (\phi^0)^3))_0 = 0$ it
follows that
\begin{equation}
\displaystyle{\int_0^1 \left[{-\frac{d}{dr}{{(\phi^0_r r)}}(\phi^0)^3}+ {{(\phi^0)^4}\over{r}}
-\gamma {(\phi^0)^4} r \right] dr =0}
\end{equation}
The latter, on by parts integration
\begin{equation*}
\displaystyle{-{(\phi^0_r r)}(\phi^0)^3 \big\vert_{0}^{1}
+ \int_0^1 3 {(\phi^0_r)^2}  {(\phi^0)^2 } r dr  + 
 \int_0^1\left[{{(\phi^0)^4}\over{r}} -\gamma_0 {(\phi^0)^4} r \right] dr =0}
\end{equation*}
that is
$$ \int_0^1 \frac{(\phi^0)^4}{r} - \gamma_0 (\phi^0)^4r \, dr \le 0,$$
and the inequality (\ref{CF}) can be easily derived.

\noindent The statements (\ref{eqK})- (\ref{CF}), together to the
local Lipschitz condition on the operators C and D, assure (see
\cite{K} ) the existence of exactly two branch of non-zero solutions
bifurcating from the point $\gamma_0/2$. Finally we remark that the
existence of two opposite branch follows from the odd functions in
(\ref{eqK}).

 \eproof  \end{proof}

\begin{remark}\label{stability}
In order to establish the stability of the solutions to (\ref{bch}),
(\ref{eqh}) around the point $\mu_0=\gamma_0/2$, we perform  a
qualitative analysis of   the bifurcation equation to the lowest
order (see equation (\ref{loweq}) below). From (\ref{eqK}) setting
$$G(h,r,\mu)= -2\mu h + L(h, r) + C(h, r, \mu) + D(h, r, \mu)= 0 $$
and
$$ 2 \mu = \gamma_0 + \delta, \qquad |\delta| << 1,$$
assuming that each element $h \in \mathcal{H}^1(0,1)$ has the unique
representation
$$ h= \beta \phi^0 + P h,\qquad (P h, \phi^0)_0 = 0,\quad \beta \in
\mathbb{R},$$ we have
$$(G(h, r,\mu), \phi^0)_0= - \delta \beta  +
(C(\beta \phi^0, r,\mu),\phi^0)_0 + \dots$$ Moreover from
(\ref{omo}), (\ref{CF}) we can get to the simple l.o. bifurcation
equation, namely
\begin{equation} \label{loweq} - \delta \beta +
\beta^3 \bar{C}=0,\qquad \bar{C} =(C(\phi^0,r,\mu),\phi^0)_0 \ge 0.
\end{equation}
It is easy to check that:\\
for $\delta \le 0$ there is the only solution $\beta=0$ and this
solution is stable (indeed in this case: $-\delta + 3 \beta^2
\bar{C} \ge 0)$;\\
for $\delta > 0$ the trivial solution is no more stable but other
two stable solutions appear, i.e. $\beta = \pm
\sqrt{\delta/\bar{C}}$.
\end{remark}
\subsection*{Acknowledgements}
Under the financial support of G.N.F.M.-I.N.d.A.M.,  I.N.F.N. and Universit\`a di Roma 
 \textsc{La Sapienza}, Rome, Italy.
M. Chipot acknowledges  Dip. S.B.A.I., Universit\`a di Roma  \textsc{La Sapienza}, for the kind hospitality.

\bibliographystyle{plain}

\nocite{*}

\hfill %\today 

\bibliography{MGVS2020bib}

\end{document}